\documentclass[11pt,a4paper]{amsart}
\usepackage[utf8]{inputenc}
\usepackage{lmodern}
\usepackage[all]{xy}
\entrymodifiers={!!<0pt,0.7ex>+}

\usepackage[T1]{fontenc} 

\RequirePackage[english]{babel}
     
\RequirePackage{amsfonts,amsmath,amssymb}

\RequirePackage{mathrsfs}    

\RequirePackage{xspace}

\def\ie{{i.e.}\kern.3em}
\def\Ie{{I.e.}\kern.3em}

\def\cf{{cf.}\kern.3em}
\def\Cf{{Cf.}\kern.3em}

\def\loccit{\textit{loc.\kern2pt cit.}\xspace}

\def\resp{resp.\kern.3em}

 \let\NN\bbN
 \let\ZZ\bbZ

\def\calL{\mathscr{L}}

\def\SL{\mathrm{SL}}
\def\GL{\mathrm{GL}}

\DeclareMathOperator{\ch}{ch}

\newtheorem{lemma}{Lemma}[subsection] 
\newtheorem{cor}[lemma]{Corollary}
\newtheorem{prop}[lemma]{Proposition} 

\theoremstyle{definition}
\newtheorem{rem}[lemma]{Remark}

\numberwithin{equation}{subsection}

\begin{document} 

{\hfuzz=6.5pt \vfuzz=4pt

\title[Line bundles over certain flag schemes]{On the cohomology of line bundles over certain flag schemes II}

\author{Linyuan Liu and Patrick Polo}

\address{Institut de Mathématiques de Jussieu-Paris Rive Gauche\\
Sorbonne Université -- Campus Pierre et Marie Curie\\
4, place Jussieu -- Boîte Courrier 247\\
F-75252 Paris Cedex 05\\
France}

\email{linyuan.liu@imj-prg.fr \quad patrick.polo@imj-prg.fr}



\dedicatory{\it To Jens Carsten Jantzen on the occasion of his 70th birthday}

\begin{abstract} 
Over a field $K$ of characteristic $p$, let $Z$ be the incidence variety in $\mathbb{P}^d \times (\mathbb{P}^d)^*$ and let $\mathscr{L}$ 
be the restriction to $Z$ of the line bundle $\mathscr{O}(-n-d) \boxtimes \mathscr{O}(n)$, where $n = p+f$ with $0 \leq f \leq p-2$. 
We prove that $H^d(Z,\mathscr{L})$ 
is the simple $\operatorname{GL}_{d+1}$-module corresponding to the partition $\lambda_f = (p-1+f,p-1,f+1)$. 
When $f= 0$, using the first author's description of $H^d(Z,\mathscr{L})$ and Jantzen's sum formula, we obtain 
as a by-product that the sum of the monomial symmetric functions $m_\lambda$, 
for all partitions $\lambda$ of $2p-1$ less than $(p-1,p-1,1)$ in the dominance order, is the alternating sum of the Schur functions $S_{p-1,p-1-i,1^{i+1}}$ for $i=0,\dots,p-2$.  
\end{abstract}

\maketitle

\setcounter{section}{1} 

\subsection*{Introduction} This paper is an addition to the first author's paper \cite{Liu19b}. 
We now consider the group scheme $G=\SL_{d+1}$ over an arbitrary field $K$ of characteristic $p > 0$. 
For $m,n\in \NN$, set $\mu_{m,n} = m\omega_1 - (n+d)\omega_d$ and write simply $\mu_n$ instead of 
$\mu_{n,n}$. We describe the $H^i(\mu_{m,n})$ when $p>n$ (thus recovering and extending \cite{Liu19b}, Cor.~2)
and for $n = p+f$ with $0\leq f \leq p-2$ we prove that $H^d(\mu_n)$ is the simple $G$-module 
$L(\lambda_f) = L(f\omega_1 + (p-2-f)\omega_2 + (f+1)\omega_3)$. 
Further, when $f=0$ we express the character of $L(\lambda_f) $, using Jantzen sum formula, as an alternating sum of Weyl characters. 
Comparing this with the character of $H^d(\mu_p)$ given in \cite{Liu19b}, Cor.~3, 
we obtain as a by-product that 
the sum of the monomial symmetric functions $m_\lambda$, 
for all partitions $\lambda$ of $2p-1$ less than $\lambda_f$ in the dominance order, 
is the alternating sum of the Schur functions $S_{p-1,p-1-i,1^{i+1}}$ for $i=0,\dots,p-2$.

\subsection{Notation} We keep the notation of \cite{Liu19b}, except that we now consider the group scheme $G=\SL_{d+1}$ over an arbitrary field $K$ of characteristic $p > 0$ and denote simply by $P$ the maximal parabolic subgroup $P_1$. 
Let $W_P$ be the Weyl group of $(P,T)$ and define $W_Q$ similarly. 
Let $\ch V$ denote the character of a $T$-module $V$. 
For each simple root $\alpha_i$, let $s_i\in W$ be the corresponding simple reflection. 
Let $w_0$ (\resp $w_P$, \resp $w_Q$) be the longest element 
of $W$ (\resp $W_P$, \resp $W_Q$) and set $N = \ell(w_Q)$. Then $\ell(w_P) = N + d-1$. 

Let $\rho_Q$ (\resp $\rho_P$) denote the half-sum of the positive roots of $Q$ (\resp $P$). 
Then one has:  

\begin{equation} 
2\rho_Q = (2-d)(\omega_1 + \omega_d) + 2 \sum_{i=2}^{d-1} \omega_i, \qquad \qquad 
2\rho_P = (1-d) \omega_1 + 2 \sum_{i=2}^{d} \omega_i .
\end{equation} 

Recall that, since $2\rho-2\rho_Q = d(\omega_1+\omega_d)$, the dualizing sheaf on $Z=G/Q$ is $\calL(-d\omega_1 -d\omega_d)$. 
Hence, by Serre duality on $G/Q$, one has 
\begin{equation}\label{SerreQ} 
H^d(\mu_{m,n}) \simeq H^{d-1}(-(m+d)\omega_1 + n \omega_d)^*. 
\end{equation} 
Further, let $\tau$ be the (involutive) automorphism of $(G,T)$ induced by the automorphism of the Dynkin diagram which swaps 
$\alpha_i$ and $\alpha_{d+1-i}$ for $i=1,\dots, d$. Note that $\tau$ also acts on $X(T)$ and preserves $X(T)^+$. 
For any $G$-module $V$, let ${}^\tau V$ denote the corresponding module twisted by $\tau$. For example, 
for a Weyl module $V(\lambda)$ (\resp a simple module $L(\lambda)$), one has ${}^\tau V(\lambda) \simeq 
V(\tau\lambda)$ (\resp ${}^\tau L(\lambda) \simeq L(\tau\lambda)$). Then, \eqref{SerreQ} can be rewrited as: 
\begin{equation}\label{symdiag}  
H^d(\mu_{m,n}) \simeq H^{d-1}(\tau\mu_{n,m})^*. 
\end{equation} 

Recall (see \cite{RAG}, II.4.6 and its proof) that if $P' \subset P''$ are parabolic subgroups containing $B$ and if $V$ is a $P'$-module, one has for all $i\geq 0$: 
\begin{equation}\label{Kempf} 
H^i(P''/B, V) \simeq H^i(P''/P', V). 
\end{equation} 
This will be used several times, without always mentioning it. 

\subsection{A description of $H^i(\mu_{m,n})$} 
For each $i\in \NN$, we denote by $H_Q^i({-})$ 
the functor $H^i(Q/B,{-})$ and define $H^i_P({-})$ similarly. Since $-\mu_{m,n}$ is trivial and 
$\mu_{m,n} - 2\rho_Q$ is anti-dominant with respect to the Levi subgroup of $Q$ one has: 
\begin{equation}\label{HQ0}
H^i_Q(-\mu_{m,n}) \simeq  
\begin{cases} 
-\mu_{m,n} & \text{ if } i = 0, 
\\
0 & \text{ if } i> 0 
\end{cases} 
\end{equation} 
and, using Serre duality on $Q/B$: 
\begin{equation}\label{HQN}
H^i_Q(\mu_{m,n} - 2\rho_Q) \simeq 
\begin{cases} 
0 & \text{ if } i < N , 
\\
H^0_Q(-\mu_{m,n})^* \simeq \mu_{m,n} & \text{ if } i = N.  
\end{cases} 
\end{equation} 

\medskip 
Consider induction from $B$ to $P$. Thanks to \eqref{HQN}, the spectral sequence of composite functors 
(see \cite{RAG}, I.4.5) degenerates and gives (together with \eqref{Kempf} applied to $P' = Q$ and $P''=P$) isomorphisms for each $i\geq 0$: 
\begin{equation}\label{N-shift}
H^i_P(\mu_{m,n}) \simeq H^{i+N}_P(\mu_{m,n}-2\rho_Q)  
\end{equation} 
and since $\mu_{m,n}-2\rho_Q = (m+d-2)\omega_1 - 2 \sum_{j=2}^{d-1} \omega_j - (n+2) \omega_d$ is anti-dominant with respect 
to $P$, the latter group is zero unless $i+N = \dim(P/B)$, \ie $i = d-1$. 
Moreover, by Serre duality on $P/B$, one has: 
\begin{align*}
H^{N+d-1}_P(\mu_{m,n}-2\rho_Q) &  \simeq H^0_P( -\mu_{m,n} + 2\rho_Q- 2\rho_P)^* 
\\ 
& = H^0_P((1-m)\omega_1 + n \omega_d)^*. 
\end{align*} 
Set $\nu_{m,n} = (1-m)\omega_1 + n \omega_d$ and $\pi_{m,n} = -w_P \nu_{m,n}$. 
Since $w_P = s_2\cdots s_d w_Q$, one has 
\begin{equation} 
\pi_{m,n} = (m-n-1)\omega_1 + n\omega_2
\end{equation} 
and since $H^0_P(\nu_{m,n})^*$ is isomorphic with $V_P(\pi_{m,n})$, 
the Weyl module for $P$ with highest weight $\pi_{m,n}$, one obtains:  

\begin{lemma} For each $m,n\in \NN$ and $i\in \NN$, one has 
$$
H^i_P(\mu_{m,n}) \simeq 
\begin{cases} 
V_P(\pi_{m,n}) & \text{ if } i = d-1, 
\\
0 & \text{ if } i \not= d-1. 
\end{cases} 
$$
\end{lemma} 

\smallskip 
Now, consider induction from $B$ to $G$. Thanks to the lemma, the spectral sequence of composite functors degenerates and gives 
(together with \eqref{Kempf} applied to $P' = P$ and $P''=G$)
isomorphisms $H^i(\mu_{m,n}) \simeq H^{i-d+1}(V_P(\pi_{m,n}))$ for each $i\geq 0$. Since the former 
is zero for $i\not\in \{d-1, d\}$ (see \cite{Liu19b}, Section 2), this gives: 

\begin{prop}\label{H0-H1}
For each $m,n\in \NN$ and $i\in \NN$, one has 
$$
H^i(\mu_{m,n}) \simeq 
\begin{cases} 
H^0(V_P(\pi_{m,n} )) & \text{ if } i = d-1, 
\\
H^1(V_P(\pi_{m,n} )) & \text{ if } i = d, 
\\
0 & \text{ if } i \not= d-1, d. 
\end{cases} 
$$
\end{prop} 

This allows us to recover and extend the result of \cite{Liu19b}, Corollary 2: 

\begin{cor}\label{p>n} 
Suppose that $p > n$. 
\begin{enumerate}
\item[(i)] If $m\geq n$, then $H^d(\mu_{m,n}) = 0$ and 
$H^{d-1}(\mu_{m,n}) \simeq H^0(\pi_{m,n})$. 
In particular, $H^i(\mu_n) = 0$ for all $i \geq 0$. 

\smallskip 
\item[(ii)] If $n\geq m$, then $H^{d-1}(\mu_{m,n}) = 0$ and 
$$
H^{d}(\mu_{m,n}) \simeq H^1(\pi_{m,n}) \simeq H^0(\pi_{n,m}) .
$$

\smallskip 
\item[(iii)] Furthermore, if\, $p > m \geq n$ then $H^{d-1}(\mu_{m,n}) \simeq L(\pi_{m,n}) \simeq H^d(\mu_{n,m})$. 
\end{enumerate} 
\end{cor} 

\begin{proof} Suppose $p > n$.  
Then $V_P(\pi_{m,n})$ is irreducible as a $P$-module, 
being isomorphic as a module over the Levi subgroup of $P$ 
to the $\GL_d$-module $S^n K^d$ which is irreducible since $n < p$. 
Therefore, $V_P(\pi_{m,n}) \simeq H^0_P(\pi_{m,n})$. 
Thus, the proposition gives for all $i\geq 0$ that 
$$
H^i(\mu_{m,n}) \simeq H^{i-d+1}(\pi_{m,n}).
$$

If $m\geq n$ then $\pi_{m,n}$ belongs to $\mathscr{C} := X(T)^{+}-\rho$ and hence $H^j(\pi_{m,n}) = 0$ for $j > 0$ (and also for $j=0$ if $m=n$). This proves (i). 

If $m <n$, then $H^0(\pi_{m,n}) = 0$. Further, $\pi_{n,m} = (n-m-1)\omega_1 + m \omega_2$ belongs to $\mathscr{C}$ and one has 
$\pi_{m,n} = s_{\alpha_1}\cdot \pi_{n,m}$. 
Since $0\leq (\pi_{n,m}+\rho, \alpha_1^\vee) = n-m < p$, one has 
$H^1(\pi_{m,n}) \simeq H^0(\pi_{n,m})$ (see \cite{RAG}, II.5.4), which proves (ii). 

\smallskip 
Suppose now that $p > m\geq n$. By (i) and Serre duality on $G/Q$, one has 
$$
H^0(\pi_{m,n}) \simeq H^{d-1}(\mu_{m,n}) \simeq H^d(-(m+d)\omega_1 + n\omega_d)^* = H^d(\tau \mu_{n,m})^* . 
$$
Now, using the automorphism $\tau$ of $(G,T)$, 
one deduces from (ii) that 
$$
H^d(\tau \mu_{n,m}))^* \simeq H^0(\tau \pi_{m,n})^*. 
$$
Since the latter is the Weyl module $V(\pi_{m,n})$, one obtains 
that $H^0(\pi_{m,n}) \simeq L(\pi_{m,n})$. This proves (iii). 
\end{proof} 

Our goal in the next subsection is to determine $H^d(\mu_p)$. Since for $d=2$ the $\SL_3$-modules 
$H^2(m,-n-2)$ have been described in \cite{Liu19a}, where it is proved in particular (see \cite{Liu19a}, Th.\,3) 
that $H^2(p,-p-2)$ is the Weyl module $V(0,p-2)$, which is simple, we will henceforth assume that $d\geq 3$. 

\subsection{Computation of $H^d(\mu_n)$ for $p\leq n \leq 2p-1$}\label{sec-Hd-simple} 
Suppose now that $d\geq 3$ and $n=p+f$ with $0\leq f \leq p-1$. 
For each weight $\nu$ which is dominant (\resp dominant with respect to $P$), denote by $L(\nu)$ (\resp $L_P(\nu)$) the irreducible $G$-module (\resp $P$-module) with highest weight $\nu$. Let 
\begin{equation}\label{def-lambda0}
\lambda_f = 
\begin{cases} 
f\omega_1 + (p-2-f)\omega_2 +(f+1) \omega_3 & \text{ if } f\leq p-2, 
\\
(p-1)\omega_1 + (p-2)\omega_3 + \omega_4 & \text{ if } f = p-1 
\end{cases} 
\end{equation}
(with the convention $\omega_4 = 0$ if $d=3$), and set $L_f = L(\lambda_f)$ and $N_f = L_P(\lambda_f)$. 

\begin{prop}\label{main}  
One has $H^{d-1}(\mu_n)\simeq L_f \simeq H^d(\mu_n)$. 
\end{prop}
\begin{rem}
  This generalizes Corollary 3 in \cite{Liu19b}, which was the case \( f=0 \). Note that if \( f>0 \) and
\( d > p-f+1 \) then \( \mu_{p+f} \) does not belong to the closure of the
facet containing \( \mu_{p} \). Indeed, denoting by \( \alpha_1, \dots,
\alpha_d \) the simple roots and setting \( \beta = \alpha_1 + \cdots
+ \alpha_{p-f+1} \), one has \( \langle\mu_{p}+\rho, \beta^{\vee}\rangle = 2p-f+1 \) whereas
\( \langle\mu_{p+f}+\rho,\beta^{\vee}\rangle = 2p+1 \).  Hence the result for \( \mu_{p+f} \) cannot be deduced from the one for \(
\mu_{p}\) by applying
a functor of translation.
\end{rem}

\begin{proof} 
Set $\pi_n = -\omega_1 + n \omega_2$ and $M = V_P(\pi_n)$. 
According to Doty \cite{Do85}, \S\S 2.3--2.4, one has exact sequences of $P$-modules:  
\begin{equation}\label{N0-M}
\xymatrix{
0 \ar[r] & N_f \ar[r] & M \ar[r] & C \ar[r] & 0
} 
\end{equation} 
and 
\begin{equation}\label{HP0-N0}
\xymatrix{
0 \ar[r] & C \ar[r] & H^0_P(\pi_n) \ar[r] & N_f \ar[r] & 0
} 
\end{equation} 
where $C = L_P(\pi_n)$. Applying the functor $H^0$ to \eqref{HP0-N0} and using that: 
\begin{equation} 
H^i(H^0_P(\pi_n)) = H^i(\pi_n) = 0 
\end{equation} 
for all $i\geq 0$, one obtains $H^0(C) = 0$ and isomorphisms $H^i(C) \simeq H^{i-1}(N_f)$ for all $i\geq 1$. 
Taking this and Proposition \ref{H0-H1} into account and applying the functor $H^0$ to \eqref{N0-M}, one obtains 
isomorphisms: 
\begin{equation}\label{H0-M}
H^{d-1}(\mu_n) \simeq H^0(M) \simeq H^0(N_f) , 
\end{equation} 
an exact sequence: 
\begin{equation}\label{H0H1-M-a}
\xymatrix{
0\ar[r] & H^1(N_f) \ar[r] & H^d(\mu_n) \ar[r] & H^0(N_f) \ar[r] & H^2(N_f) \ar[r] & 0
}
\end{equation} 
and isomorphisms 
\begin{equation}\label{H1-N0}
H^i(N_f) \simeq H^{i+1}(C) \simeq H^{i+2}(N_f)
\end{equation} 
for $i\geq 1$. Since $H^i(N_f) = 0$ for $i > \vert R^+ \vert$, one obtains $H^i(N_f) = 0$ for all $i\geq 1$. Together with 
\eqref{H0H1-M-a} and \eqref{H0-M}, this gives isomorphisms: 
\begin{equation}\label{H0H1-M-b}
H^d(\mu_n)\simeq H^0(N_f) \simeq H^{d-1}(\mu_n).
\end{equation} 
On the other hand, by \cite{Liu19b}, Cor.~4, $\lambda_f$ has multiplicity $1$ in $H^d(\mu_n)$, which is therefore non-zero. 
Thus $H^d(\mu_n)\simeq H^{d-1}(\mu_n) \simeq H^0(N_f)$ is a non-zero submodule 
of $H^0(H^0_P(\lambda_f)) = H^0(\lambda_f)$. 

Now, using the automorphism $\tau$ of $(G,T)$, one obtains that $H^d(\mu_n) \simeq H^{d-1}(\tau\mu_n)^*$ is a quotient of $H^0(\tau\lambda_f)^* \simeq V(\lambda_f)$.  
Since any non-zero morphism 
$V(\lambda_f) \to H^0(\lambda_f)$ has image $L_f$ (see \cite{RAG} II.6.16 \emph{Remark}), then \eqref{H0H1-M-b} gives that 
$H^d(\mu_n)\simeq L_f \simeq H^{d-1}(\mu_n)$. 
\end{proof}

\begin{rem} \footnote{We are grateful to one of the referees for this remark.} 
Seitz has shown (\cite{Sei87}, Prop.~6.1) that if a simple
$\SL_{d+1}$-module $L(\mu)$ with \( \mu \) \( p \)-restricted has 
one-dimensional weight spaces then either $\mu$ is a fundamental weight $\omega_i$ or a multiple of $\omega_1$ or $\omega_d$, 
or $\mu = a \omega_i + (p-1-a) \omega_{i+1}$ for some $i \in \{1,\dots, d-1\}$ and 
$a\in \{0,\dots, p-1\}$. Since $\lambda_f$ is not in that list when $f > 0$, it follows that $H^d(\mu_{p+f}) \simeq L(\lambda_f)$ 
never has all its weight spaces of dimension $1$ when $f > 0$. This improves on Remark 3 (2) of \cite{Liu19b}. 
\end{rem} 

\subsection{Jantzen sum formula and consequences} 
Now, consider the case $f=0$, \ie $n=p$. Then $\lambda_0 = (p-2)\omega_2 + \omega_3$. 
We shall use Jantzen sum formula (\cite{RAG}, II.8.19) to express $\ch L_0$ in terms of Weyl characters. 
Set $r = \min(d,p)$. In addition to $\lambda_0$, define for $i = 1,\dots, r-2$ the dominant weights: \footnote{These are different from the weights $\lambda_f$ considered in section \ref{sec-Hd-simple}.}
\begin{equation} 
\lambda_i = i \omega_1 + (p-2-i)\omega_2 + \omega_{3+i} 
\end{equation} 
(with the convention $\omega_{d+1} = 0$) 
and set $L_i = L(\lambda_i)$ and $N_i = L_P(\lambda_i)$. In other words, if $d> p$ the sequence ends with 
$\lambda_{p-2} = (p-2)\omega_1 + \omega_{p+1}$ 
whilst if $d \leq p$ it ends with 
$$
\lambda_{d-2} = (d-2)\omega_1 + (p-d) \omega_2 + \omega_{d+1} = (d-2)\omega_1 + (p-d) \omega_2.
$$

\begin{lemma}\label{wdot-0} Consider $\SL_{n+1}$ for some $n\geq 2$ and for $k = 1,\dots,n$ consider the weight 
$\theta_k = \omega_1 - k \omega_k + (k-1)\omega_{k+1}$. Note that $\theta_1 = 0$. 
\begin{enumerate} 
\item[(i)] For $k\geq 2$, one has $s_k\cdot \theta_k = \theta_{k-1}$. 

\smallskip 
\item[(ii)] For $k\geq 2$, one has $\omega_1 - k \omega_k + k \omega_{k+1} = s_k\cdots s_2 \cdot \omega_{k+1}$. 
\end{enumerate} 
\end{lemma} 

\begin{proof} One has $s_k\cdot \theta_k = \theta_k + (k-1)\alpha_k = \theta_{k-1}$. Thus 
$s_2\cdots s_k \cdot \theta_k = \theta_1 = 0$, whence $\theta_k =  s_k\cdots s_2 \cdot 0$. Next, for any 
weight $\lambda$ and $w\in W$, one has $w\cdot \lambda = w\lambda + w\cdot 0$, which equals 
$\lambda + w\cdot 0$ if $w\lambda = \lambda$. Applying this to $w = s_k\cdots s_2$ and $\lambda= \omega_{k+1}$ 
gives assertion (ii).  
\end{proof} 

\smallskip For $\alpha\in R^+$ let $\alpha^\vee$ be the corresponding coroot. Then, for $m\in \ZZ$, let $s_{\alpha, mp}$ be the affine 
reflection defined for all $\lambda\in X(T)$ by $s_{m,\alpha}(\lambda) = \lambda - ((\lambda,\alpha^\vee) - m)\alpha$. Further, one sets 
$s_{\alpha,m}\cdot \lambda = s_{m,\alpha}(\lambda+\rho)-\rho$. 

For each Weyl module $V(\lambda)$, Jantzen has defined a decreasing filtration 
$V(\lambda) \supset V(\lambda)_1  \supset V(\lambda)_2  \supset  \cdots$ and one has the following character formula (\cite{RAG}, II.8.19): 
$$
\sum_{i\geq 1} \ch V(\lambda)_i = \sum_{\alpha\in R^+} \; \sum_{\substack{m\\ 0 < mp < (\lambda+\rho,\alpha^\vee) }} 
v_p(m p) \chi(s_{\alpha, mp} \cdot \lambda) 
$$
where $v_p$ denotes the $p$-adic valuation and $\chi$ is the Weyl character. 
Recall that $\chi(\mu)=0$ if $\mu$ is singular for 
the dot action of $W$, \ie if there exists $\alpha\in R^+$ such that $(\mu+\rho, \alpha^\vee) = 0$, and otherwise there exists 
a unique couple $(w, \mu^+)\in W \times X(T)^+$ such that $\mu = w\cdot \mu^+$ and then 
$\chi(\mu) = (-1)^{\ell(w)} \ch V(\mu^+)$. 
In particular, if the right-hand side, to which we shall refer as ``Jantzen's sum'' (relative to $\lambda$), equals $\ch L$ 
for some simple module $L$ (\resp equals $0$), then one has $\ch V(\lambda) = \ch L(\lambda) + \ch L$ 
(\resp $V(\lambda) = L(\lambda)$).

\begin{prop}\label{char} For $i=0,\dots, r-2$ one has the equality: 
\begin{equation}\label{sum}  
\sum_{\ell \geq 1} \ch V(\lambda_i)_\ell = \sum_{j = i+1}^{r-2} (-1)^{j -i-1} \ch V(\lambda_j)
\end{equation} 
and the exact sequences:  
\begin{equation}\label{struc-Vi}
\xymatrix{ 
0\ar[r] & L_i \ar[r] & H^0(\lambda_i) \ar[r] & L_{i+1} \ar[r]  & 0} 
\end{equation} 
and 
\begin{equation}\label{struc-VPi}
\xymatrix{ 
0\ar[r] & N_i \ar[r] & H_P^0(\lambda_i) \ar[r] & N_{i+1} \ar[r]  & 0} 
\end{equation} 
with the convention $L_{r-1} = 0 = N_{r-1}$.
Therefore, one has exact sequences:
\begin{equation}
  \label{eq:61711d3630c0f9c9}
  0\to L(\lambda_{0})\to H^{0}(\lambda_{0})\to
  H^{0}(\lambda_{1})\to\cdots\to H^{0}(\lambda_{r-2})\to 0
\end{equation}
and
\begin{equation}
  \label{eq:b94ca9b7d1fe86f8}
  0\to L_{P}(\lambda_{0})\to H_{P}^{0}(\lambda_{0})\to
  H_{P}^{0}(\lambda_{1})\to\cdots\to H_{P}^{0}(\lambda_{r-2})\to 0.
\end{equation}
\end{prop} 

\begin{proof} Fix $i \in \{0,\dots, r-2\}$ and 
for $1\leq j \leq k \leq d$, set $\alpha_{j,k} = \alpha_j + \cdots + \alpha_k$ and 
$c_{i,j,k} = (\lambda_i +\rho, \alpha_{j,k}^\vee )$. We are going to prove that all terms in Jantzen's sum (for both $G$ and $P$) 
are zero, except the ones given in \eqref{sum}. Fix $j \leq k$ such that $c_{i,j,k} > p$, let $m\in \NN^*$ such that $mp < c_{i,j,k}$ and set 
$t = c_{i,j,k} - mp$ and $\nu_m = s_{\alpha_{j,k}, mp}\cdot \lambda_i = \lambda_i - t \alpha_{j,k}$. 
There are four cases to consider. 

\medskip 
\underline{\emph{Case} 1}: $c_{i,j,k} = k-j+1$ (this occurs only $j > i+3$ and $k \geq j+p$). 
Then the expression of $\nu_m$ in terms of the fundamental weights contains 
the ``sequence'' $-t\omega_j - t\omega_k$, the coefficients of the $\omega_\ell$ for $j< \ell < k$ being $0$,  
and hence $\nu_m + \rho$ is orthogonal to both $\alpha_{j,j+t-1}^\vee$ and $\alpha_{k-t+1,k}^\vee$. Therefore 
$\nu_m$ gives no contribution to Jantzen's sum, neither for $G$ nor for $P$. 

\medskip 
\underline{\emph{Case} 2}: $c_{i,j,k} = k-j+2$ (this occurs only if $j\leq i+3 \leq k$, including the case $k=i+3$, $j=2$ and $i = p-2$). 
Assume first that $j < i+3 < k$. 
Then the expression of $\nu_m$ in terms of the fundamental weights contains 
the ``sequence'' $-t\omega_j + \omega_{i+3}- t\omega_k$, the coefficients of the $\omega_\ell$ for $j< \ell < k$ and $\ell\not= i+3$ 
being $0$. 

For $s = 0,\dots, i+2-j, i+3-j, \dots, k-j-1$, $(\nu_m+\rho, \alpha_{j,j+s}^\vee)$ takes all the values 
from $1-t$ to $k-j+1-t$, except $i-j+4-t$. Since 
$1-t \leq 0$ and $k-j+1-t > 0$ (since $t \leq k-j+2-p \leq k-j$), the value $0$ is obtained unless $t = i-j+4$. 

Similarly, for $s=0,\dots,k-i-4,k-i-3,\dots,k-j-1$, $(\nu_m+\rho, \alpha_{k-s,k}^\vee)$ takes all the values from 
$1-t$ to $k-j+1-t$, except $k-i-2-t$. Hence 
the value $0$ is obtained, unless  $t = k- i -2$. This shows that $\nu_m+\rho$ is singular, except possibly if 
$t = i-j+4 = k - i-2$. But in this case one has $2t = k-j+2$ and hence:  
$$
(\nu_m + \rho, \alpha_{j,k}^\vee) = k-j + 2(1-t) = 0.
$$
Consider now the ``boundary'' cases $j=i+3$ or $k = i+3$. 
If $j=i+3$ then, since $k-t+1 > j$, one has $k-t+1 > i+3$ and hence $(\nu_m + \rho, \alpha_{k-t+1,k}^\vee) = 0$. 
If $ k = i+3$, then $j + t-1 < k = i+3$ and hence $(\nu_m + \rho, \alpha_{j,j+t-1}^\vee) = 0$. 
Thus, in any case $\nu_m$ gives no contribution to Jantzen's sum, neither for $G$ nor for $P$. 

To close this case, note that  since $i+3 \leq p +1 \leq k - j +2$, the case $k = i+3$ can occur only if $j \leq 2$, in which case $c_{i,j,k} = k-j+2$ implies $i = p-2$. 

\medskip 
\underline{\emph{Case} 3}: $j = 2$ and $c_{i,j,k} = k+p-2-i$ and $i < p-2$. (Note that   $c_{i,2,i+2} = p-1$ hence the hypothesis $c_{i,2,k} > p$ implies 
$k\geq i+3$.) 
Then one has $\nu_m = (i+t)\omega_1 + (p-2-i-t)\omega_2 + \omega_{i+3} -t \omega_k + t\omega_{k+1}$. 

For $s=0,\dots,i,i+1,\dots,k-3$, $(\nu_m+\rho,\alpha_{2,2+s}^\vee)$ takes all the values from 
$p-i-1-t$ to $p+k-i-3-t$, except $p-t$. 
Since the last value taken is $\geq p-1> 0$, the value $0$ is obtained except if $t = p$ or if the initial value $p-i-1-t$ is $>0$, 
\ie $t \leq p-i-2$. 

Similarly, for $s= 0, \dots, k-i-4, k-i-3,\dots, k-3$, $(\nu_m+\rho,\alpha_{k-s,k}^\vee)$ takes all the values from 
$1-t$ to $k-1-t$, except $k-i-2-t$. 
Moreover one has $1-t \leq 0 < k-1-t$, hence the value $0$ occurs unless  $t = k-i-2$. 

\smallskip 
Thus, $\nu_m+\rho$ is singular except possibly if $t = k-i-2$ belongs to $\{1,\dots, p-i-2\}$ or if $t = k-i-2 = p$. In the latter case, one 
has $2t = p+k-i-2$ and hence 
$$
(\nu_m + \rho, \alpha_{2,k}^\vee) = k-2 + p-i-2t = 0.
$$
In the former case, one has $k = i+2+t$ with $t = 1,\dots, p-i-2$, whence $m=1$. In fact, since $k = i+2+t$ is $\leq d$, we have 
$t \in \{1,\dots, r-i-2\}$, recalling that $r = \min(p,d)$. For $t = 1,\dots, r-i-2$, set 
$$
\theta'_t = s_{\alpha_{2,i+2+t},p} \cdot \lambda_i = (i+t)\omega_1 + (p-2-i-t)\omega_2 + \omega_{i+3} -t \omega_{i+2+t} + 
t\omega_{i+3+t}.
$$
Using Lemma \ref{wdot-0} with a shift of $i+2$ in the indices, one obtains that $\theta'_1 = \lambda_{i+1}$ and that  
$\theta'_t = s_{i+2+t}\cdots s_{i+4}\cdot \lambda_{i+t}$ for $t\geq 2$. 

\smallskip 
Denote by $G_P$ the Levi subgroup of $P$ containing $T$ and recall (\cite{RAG}, II.5.21) that 
$V_P(\lambda_i)$ is just the corresponding Weyl module for $G_P$ on which the unipotent radical of $P$ acts trivially. 
Therefore, applying Jantzen's sum formula for $G_P$, one already obtains the equality: 
\begin{equation}\label{sum-P}  
\sum_{\ell \geq 1} \ch V_P(\lambda_i)_\ell = \sum_{t = 1}^{r-i-2} (-1)^{t -1} \ch V_P(\lambda_{i+t}). 
\end{equation} 
To prove the analogous equality for $G$ we must consider the last case, where $j=1$.

\medskip 
\underline{\emph{Case} 4}: $j=1$. Note that the assumption $c_{i,1,k} > p$ implies $k \geq 3$. If $k \leq i+3$ then 
$c_{i,1,k} = p + k-2 + \delta_{k,i+3}$ is $\leq 2p$ (since $i\leq p-2$), hence $m=1$ and $t = k-2 + \delta_{k,i+3}$ and the expression 
of $\nu_m$ 
in terms of the fundamental weights contains 
the ``sequence'' $(p-i-2)\omega_2 - (k-2)\omega_k$, the coefficients of the $\omega_\ell$ for $2 < \ell < k$  
being $0$. Then $(\nu_m+\rho, \alpha_{3,k}^\vee) = 0$ hence $\nu_m$ gives no contribution to Jantzen's sum. 

Suppose now that $k > i+3$. Then $c_{i,1,k} = k-1+p$ and 
$$
\nu_m = (i-t)\omega_1 + (p-2-i)\omega_2 + \omega_{i+3} - t \omega_k + t\omega_{k+1}.
$$
For $s = 0,\dots, k-i-4,k-i-3,\dots,k-3$,  $(\nu_m+\rho,\alpha_{k-s,k}^\vee)$ takes all the values from 
$1-t$ to $k-1-t$, except $k-i-2-t$. 
Moreover one has $1-t \leq 0 \leq k-1-t$, hence the value $0$ occurs unless  $t = k-i-2$. 

\smallskip 
Let us assume henceforth that $t = k-i-2$. Then $t > p$, for otherwise one would have $k-i-2 \leq p$ hence 
$k \leq 2p$ (since $i\leq p-2$) whence $m\leq 2$; but $m=1$ gives $k-i-2 = t = c_{i,1,k} - p = k-1$, a contradiction, 
whereas $m=2$ gives $k-i-2 = t = c_{i,1,k} - 2p = k-p-1$, a contradiction too, since $i\leq p-2$. 

Now, for $s= 0,\dots,i,i+1,\dots,k-3$, $(\nu_m+\rho,\alpha_{1,2+s}^\vee)$ takes all the values from 
$p-t$ to $p+k-2-t$, except $p+i+1-t$. 
Moreover since $p < t \leq k-1$, the initial term is $< 0$ and the final term $>0$, hence the value $0$ occurs unless  $t = p+i+1$. 
Now, if $t = k-i-2 = p+i+1$ then $2t = p+k-1$ and hence 
$$
(\nu_m + \rho, \alpha_{1,k}^\vee) = p+k-1 - 2t = 0.
$$
Thus, in any case $\nu_m$ gives no contribution to Jantzen's sum. This proves \eqref{sum}. 

\medskip 
It follows from \eqref{sum} that $H^0(\lambda_{r-2}) = L_{r-2}$; then for $\lambda_{r-3}$ the Jantzen sum equals 
$\ch L_{r-2}$ hence $\ch H^0(\lambda_{r-3} ) = \ch L_{r-3}+ \ch L_{r-2}$. By 
decreasing induction one obtains that $\ch H^0(\lambda_{i}) = \ch L_{i}+ \ch L_{i+1}$ for $i = r-3,\dots, 0$, whence 
the exact sequences \eqref{struc-Vi}. Similarly, 
using \eqref{sum-P} one obtains that $H^0_P(\lambda_{r-2}) = N_{r-2}$ and 
$\ch H^0_P(\lambda_{i}) = \ch N_{i}+ \ch N_{i+1}$ for $i = r-3,\dots, 0$, whence the exact sequences \eqref{struc-VPi}. 
This completes the proof of Proposition \ref{char}. 
\end{proof} 

Let us derive the following corollary (which is not used in the sequel). 

\begin{cor}\label{H0-Ni}  
For $i = 0,\dots, r-2$, one has $H^0(N_i) = L_i$ and $H^j(N_i)$ $= 0$ for $j>0$. 
\end{cor} 

\begin{proof} Applying the functor $H^0$ to each exact sequence \eqref{struc-VPi} 
gives an exact sequence:  
$$
\xymatrix{ 
0\ar[r] & H^0(N_i) \ar[r] & H^0(\lambda_i) \ar[r] & H^0(N_{i+1}) \ar[r]  & H^1(N_i) \ar[r] & 0} 
$$
and isomorphisms $H^{j}(N_{i+1}) \simeq H^{j+1}(N_i)$ for $j \geq 1$. Since $H^j(N_0) = 0$ for $j\geq 1$, one obtains 
$H^j(N_i) = 0$ for all $i\geq 0$ and $j\geq 1$, hence the previous exact sequence becomes: 
\begin{equation} 
\xymatrix{ 
0\ar[r] & H^0(N_i) \ar[r] & H^0(\lambda_i) \ar[r] & H^0(N_{i+1}) \ar[r]  &  0} . 
\end{equation} 
Since $H^0(N_0)\simeq L_0$, the exact sequences \eqref{struc-Vi} then imply, by induction on $i$, that 
$H^0(N_i)\simeq L_i$ for $i=0,\dots, r-2$. 
\end{proof}

On the other hand, in \cite{Liu19b}, Cor.\,3 and 4, the first author
proved, using a result of Suprunenko pointed out by one of the referees, that 
the dominant weights of $L(\lambda_{0})$ (resp. of \( L(\lambda'_{0})=L((p-2)\omega_{1}+\omega_{2}) \)) are exactly the dominant weights 
$\leq \lambda_0$ (resp. \( \lambda'_{0} \)), each occuring with multiplicity one. 
Let us now switch to representations of $\GL_{d+1}$ and identify each
$\lambda_i$ (resp. \( \lambda'_i = (p-2-i)\omega_1 + \omega_{2+i}  \)) with the partition 
$(p-1,p-1-i,1^{i+1})$ (resp. \( (p-1-i, 1^{i+1}) \)). Recall that (see for example \cite{McD}, Chap.\,I) for a dominant weight $\lambda$ of $\GL_{d+1}$, identified with a partition with at most $d+1$ parts, 
the Weyl character $\ch V(\lambda)$ (\resp the orbit sum $\sum_{\nu\in W\lambda} e^\nu$) corresponds to the Schur function 
$S_\lambda$ (\resp the monomial symmetric function $m_\lambda$). 
Let us assume that $d \geq 2p-2$, then the dominant weights smaller than $\lambda_0$ correspond to all partitions of $2p-1$ 
which are smaller than $(p-1,p-1,1)$ in the dominance order. Thus, one
deduces from Propositions  \ref{char} and \ref{main} (or Corollaries 3 and 4 
in \cite{Liu19b}) the following: 

\begin{cor}\label{alt-Schur-p} For each prime number $p$, one has the equalities: 
$$
\sum_{\lambda \leq (p-1,p-1,1) } m_\lambda = \sum_{i=0}^{p-2} (-1)^i \, S_{ (p-1,p-1-i,1^{i+1} ) } 
$$
and
\begin{displaymath}
 \sum_{\lambda \leq (p-1,1) } m_\lambda = \sum_{i=0}^{p-2} (-1)^i \, S_{ (p-1-i,1^{i+1} ) }  
\end{displaymath}
where the sum on the left hand-side of the first equality
(resp. second equality) is taken over all partitions $\lambda$ of\,
$2p-1$ (resp. \( p \)) such that $\lambda \leq (p-1,p-1,1)$ 
 (resp. \( \lambda\leq (p-1,1) \)) in the dominance order. 
\end{cor} 

\begin{rem} One may conjecture that the previous equalities hold for \emph{any} integer $n \geq 2$ (not only for prime numbers). 
\footnote{On August 23, 2019 (the day after the first version of this article was
  posted on arXiv), a proof of this conjecture for the first equality was kindly sent to us by Darij Grinberg, see \cite{Gri20}.} 
\end{rem} 

\medskip
{\bf Acknowledgements.} We thank one of the referees for several very useful comments, in particular for suggesting the second 
equality of Corollary \ref{alt-Schur-p}.

}
\end{document}